\documentclass[12pt]{article}
\usepackage{amsmath, amssymb, amsfonts, amsthm, enumerate, float, lineno, mathrsfs, subcaption, tikz}
\usepackage{longtable,fullpage}
\DeclareMathOperator{\av}{Av}
\let\SS\relax
\DeclareMathOperator{\SS}{Seg}
\renewcommand{\S}{\mathcal{S}}
\newcommand{\C}{\mathcal{C}}
\newcommand{\N}{\mathcal{N}}
\newcommand{\W}{\mathcal{W}}
\newcommand{\e}{\mathbf{e}}

\theoremstyle{plain}
\newtheorem{theorem}{Theorem}
\newtheorem{lemma}[theorem]{Lemma}
\newtheorem{cor}[theorem]{Corollary}
\newtheorem{prop}[theorem]{Proposition}

\theoremstyle{definition}

\newtheorem{example}[theorem]{Example}
\newtheorem{conjecture}[theorem]{Conjecture}

\newcommand{\ds}{\displaystyle}

\title{Enumeration of cyclic permutations in vector grid classes}
\author{Kassie Archer and L.-K. Lauderdale}
\date{}

\begin{document}

\maketitle

\begin{abstract}
A grid class consists of permutations whose pictorial depiction can be partitioned into increasing and decreasing parts as determined by a given matrix. In this paper, we introduce a method for enumerating cyclic permutations in vector grid classes by establishing a bijective relationship with certain necklaces. We use this method to complete the enumeration of cyclic permutations in the length 3 vector grid classes. In addition, we define an analog of Wilf-equivalence between these sets. We conclude by discussing cyclic permutations in alternating grid classes.
\end{abstract}

% history:
% \received{\smonth{1} \sday{1}, \syear{0000}}

%\tableofcontents

\section{Introduction}

%\todo{cite David Bevan Thesis}

Grid classes are well-studied classes of permutations (see \cite{AAVRB, HucVat, MurVat} for example) that are assigned a signature in the form of a matrix $M$ with entries in $\{-1,0,1\}$, which in some way determines the structure of the permutations. More specifically, permutations in a given grid class with signature $M$ are comprised of increasing and decreasing parts laid out in a grid that is determined by $M$. For example, a permutation in the grid class with signature $M=\begin{bmatrix} 1 & 1
\end{bmatrix}$ is comprised of an increasing segment followed by another increasing segment, and thus has at most one descent. We are primarily concerned with the vector grid classes (i.e., those where $M$ is a vector, sometimes referred to as \textit{juxtaposition classes}), which received special attention in \cite{ATKINSON_RP, AMR, ARCHERELIZALDE2014} among others.  In particular, we prove a formula regarding a relationship between cyclic permutations in vector grid classes and necklaces. We then use this formula to enumerate the cyclic permutations in the length 2 and 3 vector grid classes.

Cyclic permutations in given vector grid classes have appeared in several other papers. For example, they were used to characterize the permutations realized by the periodic points of certain dynamical systems \cite{ARCHERELIZALDE2014}, appeared in the character formula for a certain representation of the symmetric group \cite{ARCHER2016}, and were used in the analysis of card-shuffling techniques \cite{DFH2013}. %. In particular, the cyclic permutations enumerated in this paper correspond to periodic patterns of certain signed shifts, as described in \cite{ARCHER14}. 
 In addition, the descent structure of cyclic permutations, which is closely related to the grid structure, was studied in various papers \cite{gessel, ELIZALDE11, BARIL13}.
% \todo{defined vector grid classes and make changes throughout}
 
 In this paper, we subscribe to the convention used in \cite{ARCHERELIZALDE2014}, which is to use a signature $\sigma$ of $+$'s and $-$'s in place of the row vector matrix $M$ of $1$'s and $-1$'s. 
 Our goal in this paper is two-fold: (1) to establish an enumerative relationship between cyclic permutations in grid classes with signature $\sigma\in\{+,-\}^k$ and $k$-ary necklaces, and (2) to use this relationship to enumerate the set of cyclic permutations in certain grid classes, illustrating its use.
 
 In Section  \ref{SECTION BACKGROUND}, we provide the necessary background for this paper, including definitions and propositions regarding necklaces. In \cite{gessel}, Gessel and Reutenauer used a bijection between necklaces and permutations to determine the number of permutations with a given cycle type and descent set. This was modified in \cite{ARCHERELIZALDE2014} to determine the number of cyclic permutations with signature $\sigma=+^k$, $\sigma=-^k$, or $\sigma=+-$. Analysis of the cycle structure of unimodal permutations (i.e., those with signature $\sigma=+-$, can also be found in \cite{WeissRogers, Thibon, Gannon}). In Section \ref{section main theorem}, we prove the main theorem of this paper that generalizes these previous results to establish a formula relating cyclic permutations in grid classes with signature $\sigma\in\{+,-\}^k$ and $k$-ary necklaces.
 
 In Section \ref{section 2-vector}, we recover the enumerations of cyclic permutations in the four grid classes with signature $\sigma$ when $|\sigma|=2$.
 We also prove some necessary results about unimodal cyclic permutations with a given peak position, which will aid in the enumeration of cycles in length 3 vector grid classes.
 
 %\todo{do we want a table of these values too?}
 
 In Section \ref{section 3-vector},  we complete the enumeration of cyclic permutations in the eight grid classes with signature $\sigma$, where $|\sigma|=3$. The first few terms and the OEIS reference number for each sequence can be found in Table~\ref{TABLE SIGMA}. 
 %For convenience, we label these grid classes $\sigma_{1}$ through $\sigma_{8}$ according to Table \ref{TABLE SIGMA}. 

	\begin{figure}[h]
		\centering
	\begin{tabular}{|c|c|c|}
	\hline
	 $\sigma$ & First ten terms (starting with $n=1$) & OEIS entry \\ \hline \hline
	 $+++$ &  $1, 1, 2,6,18,62,186,570, 1680, 4890$ &A303117 \\ \hline
	 $---$ &$1, 1, 2,6,18,58,186,570, 1680, 4878$ & A304200 \\ \hline
	 $+-+$ & $1, 1, 2, 5, 12, 30, 78, 205, 546, 1476$ & A136704 \\ \hline
	 $-+-$ & $1, 1, 2, 5, 12, 30, 78, 205, 546, 1476$ & A136704  \\ \hline
	$++-$ & $1,1,2,5,15,42,120,338,952,2671$ &A303980\\ \hline
	$+--$ &  $1,1,2,5,15,43,120,338,952,2672$&A304201 \\ \hline
	$-++$ & $1,1,2,5,15,42,120,338,952,2671$& A303980  \\ \hline
	$--+$ & $1,1,2,5,15,43,120,338,952,2672$&A304201  \\ \hline
	\end{tabular} 
	\captionof{table}{The signatures of the eight length 3 vector grid classes, together with the the first few terms of the enumeration of cycles in the corresponding grid classes and the OEIS entries \cite{OEIS}.}
	\label{TABLE SIGMA}
	\end{figure}
%		\todo{finish table of values (with OEIS entries)}
%	
%
%	\begin{figure}[h]
%		\centering
%	\begin{tabular}{c|c||c|c||c|c||c|c}
%		$i$ & $\sigma_i$ & $i$ & $\sigma_i$ & $i$ & $\sigma_i$ & $i$ & $\sigma_i$\\ \hline
%		1& $+++$ &2 & $---$& 3 & $+-+$ &4 & $-+-$ \\
%		 5 & $++-$ &  6& $+--$ & 7 & $-++$ & 8 & $--+$\\
%	\end{tabular} 
%	\captionof{table}{The signatures of the eight $3 \times 1$ grid classes}
%	\label{TABLE SIGMA}
%	\end{figure}
%	
%
%
%In Theorem \ref{thm sigma1 and sigma 2}, we state these results for $\sigma_1$ and $\sigma_2$ (i.e., in the special case when $k=3$).

%In addition to enumerating the cyclic permutations of a given grid class, 
In Section \ref{SECTION WILF}, 
we generalize the notion of Wilf-equivalence from classical pattern avoidance to this setting. We say that two grid classes are \textit{cyc-Wilf-equivalent} if the number of cyclic permutations of length $n$ in one grid class equals the number of cyclic permutations of length $n$ in the other. Additionally, we say that two grid classes are \textit{weakly cyc-Wilf-equivalent} if their sizes are equal when $n\not\equiv 2\!\pmod 4$. The results (proven in Section~\ref{SECTION WILF}) are summarized in Table~\ref{FIGURE WILF RESULTS}. 

\begin{figure}[h]
\centering
\begin{tabular}{c||c}
cyc-Wilf-equivalence classes & weakly cyc-Wilf-equivalence classes \\ \hline
$+++$ & $+++,---$ \\
$---$ & $+-+, -+-$ \\ 
$+-+, -+-$ & $++-, +--, -++, --+$ \\
$++-, -++$ & \\
$+--, --+$ & \\
\end{tabular}
\captionof{table}{List of cyc-Wilf- and weakly cyc-Wilf-equivalence classes}
\label{FIGURE WILF RESULTS}
\end{figure}

	 In Section~\ref{SECTION ALTERNATING}, we show that the alternating grid classes of a given size are cyc-Wilf-equivalent. Finally, in Section~\ref{SECTION CONJECTURES}, we include discussion and conjectures.

%previous work in these classes (ArcEli) & involutions etc
%summary of results and outline of paper

%%%%%%%%%%%%%%%%%%%%%%%%%%%%%%%%%%%%%%%%%
%%%%%%%%%%%%%%%%%%%%%%%%%%%%%%%%%%%%%%%%%

\section{Background}\label{SECTION BACKGROUND}

The set of permutations of $[n]=\{1,2, \ldots, n\}$ is denoted $\mathcal S_n$ and we write $\pi \in \mathcal S_n$ in its one-line notation as $\pi=\pi_1\pi_2\ldots\pi_n$.  

\subsection{Permutation statistics}

We say that a permutation $\pi\in \S_n$ has a \textit{descent} at position $i$ if $\pi_i>\pi_{i+1}$. We say $\pi$ has an \textit{ascent} at position $i$ if $\pi_{i}<\pi_{i+1}$. For example, the permutation $4156732$ has three descents, namely those at positions 1, 5, and 6, and has three ascents, namely those at positions 2, 3, and 4. 

We define a \textit{peak} of the permutation $\pi\in\S_n$ to be $i\in[n]$ such that $\pi_{i-1}<\pi_i$ and $\pi_i>\pi_{i+1}$ (where $\pi_0=\pi_{n+1}:=0$). Similarly, we define a \textit{valley} of the permutation $\pi\in\S_n$ to be $i\in[n]$ such that $\pi_{i-1}>\pi_i$ and $\pi_i<\pi_{i+1}$ (where $\pi_0=\pi_{n+1}:=n+1$). For example, the permutation $4156732$ has peaks at positions 1 and 5 and has valleys at positions 2 and 7. Notice that this definition is non-standard since we allow peaks and valleys to occur at the beginning and end of the permutation.

Finally, we say a permutation is \textit{unimodal} if it has exactly one peak. Any valleys must occur at the beginning or end of the permutation. For example, there are eight unimodal permutations of length 4: $$1234, 1243, 1342, 2341, 1432, 2431, 3421, 4321.$$ These permutations respectively have their unique peak at position $4$, 3, 3, 3, 2, 2, 2, and 1.

\subsection{The $\sigma$-classes}

For $m<n$, we say a permutation $\pi\in \S_n$ \textit{contains} the pattern $\tau \in \S_m$ if there are indices $i_1<\cdots <i_m$ so that the subsequence $\pi_{i_1}\pi_{i_2} \ldots \pi_{i_m}$ is in the same relative order as $\tau$, and we say $\pi$ \textit{avoids} $\tau$ if $\pi$ does not contain it. For example, the permutation $\pi=142365$ avoids 321 since there is no length 3 subsequence of $\pi$ that is decreasing. The set of permutations that avoid a given pattern $\tau$ is denoted $\av(\tau)$ and the set of permutations that avoid a set of patterns $B = \{\tau_1, \tau_2, \ldots\}$ is denoted $\av(B)$. Any permutation class (i.e., a set of permutations closed under pattern containment) may be characterized in terms of pattern avoidance. %For any permutation class $\mathscr{C}$, we denote by $\mathscr{C}_n$ the set of permutations in $\mathscr{C}$ of length $n$.

Let $\sigma=\sigma_0\sigma_1\ldots\sigma_{k-1} \in \{+,-\}^k$ and define a partition of the set $\{0,1, \ldots, k-1\}$ by $T_{\sigma}^+=\{i:\sigma_i=+\}$ and $T_{\sigma}^-=\{i:\sigma_i=-\}$. That is, $T_{\sigma}^+$ is the set of locations of $+$ in $\sigma$, and $T_{\sigma}^-$ is the set of locations of $-$ in $\sigma$.  For example, if $\sigma=++-+$, then $T^+_\sigma=\{0,1,3\}$ and $T^-_\sigma=\{2\}$. 
%
%
%We denote by $\av(\rho)$ the set of permutations avoiding the pattern $\rho$.  
We define the $\sigma$-\textit{class}, denoted $\mathcal S^\sigma$, as in \cite{AAR, ARCHERELIZALDE2014, AMR}, to be the set of permutations comprised of $k$ contiguous (possibly empty or singleton) segments, so that the $i$-th such segment is increasing when $i\in T_\sigma^+$ and decreasing when $i \in T_\sigma^-$. For example $\S^{+-}$ is the set of unimodal permutations and $\S^{+++}$ is the set of permutations with at most two descents. A $\sigma$-class is an example of a row vector grid class as defined in \cite{AAVRB, HucVat, MurVat}. Indeed, it is exactly the length $k$ vector grid class for the matrix $M = [ M_0 \, M_1 \, \cdots \, M_{k-1}]$, where $$M_i = \begin{cases} +1 & \text{if } \sigma^i = + \\ -1 & \text{if } \sigma^i = -,\end{cases}$$ and can be drawn on a series of line segments each with slope $\pm1$, as determined by $\sigma$.
%\todo{add pictures of grid classes?}
%Indeed, it is exactly the juxtaposition $[\mathscr A_0 \ \mathscr A_1 \ \ldots \ \mathscr A_{k-1}]$ where
%	\[\mathscr A_t = \left\{\begin{array}{ll} \av(21) & \text{ if } \sigma_t=+\\  \av(12) & \text{ if } \sigma_t=-\end{array}\right.\]
%and $0 \leq t \leq k-1$.  
%For example, $\pi_1=326791854$ is an element of
%	\[\mathcal S_9^{-+-}  = [\av(12) \ \av (21) \ \av(12)] \]
%and $\pi_2=126793854$ is an element of 
%	\[\mathcal S_9^{++-}  = [\av(21) \ \av (21) \ \av(12)].\]

%%Sigma segmentations are usually referred to as "griddings" when talking about grid classes. Should we rewrite to fit more closely to the grid-class language?
Any $\pi \in \mathcal S^\sigma_n:=\S^\sigma\cap\S_n$ must admit a $\sigma$-\textit{segmentation} defined in \cite{ARCHERELIZALDE2014} (and referred to as a \textit{gridding} in \cite{AAVRB, BEVAN2014, BEVAN2015}) to be a sequence $\mathbf{e}=(e_0, e_1, \ldots, e_k)$ such that $0=e_0 \leq e_1 \leq \cdots \leq e_k=n$ and each segment $\pi_{e_i+1}\pi_{e_i+2}\cdots\pi_{e_{i+1}}$ is increasing if $i \in T_{\sigma}^+$ and decreasing if $i \in T_{\sigma}^-$.  	
For example, if $\sigma=++-$, then the permutation $\tau=268147953$ has two $\sigma$-segmentations, namely $(0, 3, 6, 9)$ and $(0,3,7,9)$. If $\sigma=-+-$, then the permutation $\tau'=862347951$ has four $\sigma$-segmentations: $(0,i,j,9)$, where $i \in \{2,3\}$ and $j \in \{6,7\}$.
The pictorial representations of the permutations $\tau$ and $\tau'$ are seen in Figure~\ref{SIGMA-CLASS EXAMPLES} below.

%\todo{illustrate sigma segmentations in these pictures?}
	\begin{figure}[H]
	\centering
		\begin{tabular}{ccc}
			\subcaptionbox{The permutation $\tau$ in $\mathcal S_{9}^{++-}$}{
			\begin{tikzpicture}[scale=.55]
				\newcommand\myx[1][(0,0)]{\pic at #1 {myx};}
				\tikzset{myx/.pic = {\draw 
 				     (-2.5mm,-2.5mm) -- (2.5mm,2.5mm)
 				     (-2.5mm,2.5mm) -- (2.5mm,-2.5mm);}}

				\draw[gray, thin] (0,0) grid (9,9);
				\foreach \x/\y in {1/2,2/6,3/8,4/1,5/4,6/7,7/9,8/5,9/3}
				\myx[(\x-.5,\y-.5)];
			\end{tikzpicture} }  & \hspace{1cm} &
			
			\subcaptionbox{The permutation $\tau'$ in $\mathcal S_{9}^{-+-}$}{
			\begin{tikzpicture}[scale=.55]
				\newcommand\myx[1][(0,0)]{\pic at #1 {myx};}
				\tikzset{myx/.pic = {\draw 
				      (-2.5mm,-2.5mm) -- (2.5mm,2.5mm)
				      (-2.5mm,2.5mm) -- (2.5mm,-2.5mm);}}

				\draw[gray, thin] (0,0) grid (9,9);
				\foreach \x/\y in {1/8,2/6,3/2,4/3,5/4,6/7,7/9,8/5,9/1}
				\myx[(\x-.5,\y-.5)];
				
				%\draw[-, thick,gray] (.5,7.5) -- (2.5,1.5);
			\end{tikzpicture} }

		\end{tabular}
		\caption{Pictorial representations of the permutations $\tau=268147953$ and $\tau'=862347951$ in $\mathcal S_{9}^{++-}$ and $\mathcal S_{9}^{-+-}$ respectively}
		\label{SIGMA-CLASS EXAMPLES}
	\end{figure}

Every $\sigma$-class is a permutation class and can therefore be characterized in terms of pattern avoidance. For example, $\S^{+-} =\av(213, 312)$. In \cite{ATKINSON_RP}, Atkinson gives a constructive proof that all permutations of a $\sigma$-class can be characterized as avoiding a finite set of patterns.

\subsection{Cyclic permutations}
Let $\mathcal C_n$ denote the set of cyclic permutations of $\mathcal S_n$ (i.e., the permutations in $\mathcal S_n$ that are composed of a single cycle with length $n$).  The map defined in \cite{ELIZALDE2009} by
	\[\begin{array}{rcl}
		\theta:\mathcal S_n & \rightarrow & \mathcal C_n\\
		\tau & \mapsto & \pi
	\end{array}\]
sends a permutation $\tau = \tau_1\tau_2\ldots \tau_n\in\S_n$ to the cyclic permutation $$\pi=(\tau_1\tau_2\ldots \tau_n) =\pi_1\pi_2\ldots\pi_n\in\C_n.$$ For example, $\theta(24513)=(24513)=34251$. Let $\bar{\S}_n$ denote the set of equivalence classes of permutations in $\S_n$ up to cyclic rotation. For example, $|\bar\S_3|=2$ since $[123]=\{123, 231, 312\}$ and $[132]=\{132, 321, 213\}$ are the only two elements. Clearly, we can define a bijection $\bar\theta:\bar\S_n \to \C_n$ by $\bar\theta([\tau])=\theta(\tau)$.

Additionally, we denote by $\C^\sigma$ the set of cyclic permutations in the $\sigma$-class (i.e., $\C_n^\sigma := \C_n\cap\S^\sigma$). 
For a signature $\sigma$, we set $c_\sigma(n):=|\mathcal C_n^{\sigma}|$. The proposition below implies that $c_{++-}(n) = c_{-++}(n)$ and that $c_{+--}(n) = c_{--+}(n)$. As a result, the signatures $++-$ and $-++$ are cyc-Wilf-equivalent, as are the signatures $+--$ and $--+$.

\begin{prop}\emph{\cite[Proposition 2.7]{ARCHERELIZALDE2014}}\label{REVERSE SAME}
For $\sigma = \sigma_0\sigma_1\ldots\sigma_{k-1}\in\{+,-\}^k$, let $\sigma^r$ denote the signature $\sigma_{k-1} \sigma_{k-2} \ldots \sigma_0$. Then $|\C_n^{\sigma}| = |\C_n^{\sigma^r}|$.
\end{prop}
 
The complementary signature $\sigma^c$, defined as $\sigma_i^c = +$ if $\sigma_i=-$ and $\sigma_i^c = -$ if $\sigma_i=+$, is not necessarily cyc-Wilf-equivalent to $\sigma$. For example, if $\sigma = +--$, then $\sigma^c = -++$ and these are not cyc-Wilf-equivalent. We conjecture in Section~\ref{SECTION CONJECTURES} that $\sigma$ and $\sigma^c$ are always weakly cyc-Wilf-equivalent.

\subsection{Words and necklaces}
%\todo{this section supposedly needs a lot of work}
A \textit{word} of length $n$ on $k$ letters is a sequence $s = s_1s_2\ldots s_n$ with $s_i \in \{0,1, \ldots, k-1\}$ for all $i\in[n]$. We denote the set of all such words by $\W_k(n)$. For example, $$\W_3(2)=\{00, 01, 02, 10, 11, 12, 20, 21, 22\}.$$ A word $s$ is \textit{$r$-periodic} if it can be written as the concatenation of $r$ copies of a word $q$ of length $\frac{n}{r}$ (i.e., if $s = q^r$), and a word is \textit{primitive} if it not $r$-periodic for any $r>1$. In other words, we say a word is primitive if it cannot be written as the concatenation of two or more copies of a shorter word. For example, the word $0020112$ is primitive and the word $00120012$ is 2-periodic. As another example, all binary words of length 7 are primitive with the exceptions of the word 0000000 and  the word 1111111.

We define the \textit{evaluation} of a word to be the sequence $(a_0, a_1, \ldots, a_{k-1})$, where $a_i$ is defined as $$a_i=|\{j \in [n] : s_j = i\}|.$$ That is, the evaluation of a word records the number of times each letter appears in the word. For example, the evaluation of the word $0012001022$ is $(5,2,3)$ since there are five 0's, two 1's and three 2's. The evaluation of $0020020$ is $(5, 0, 2)$ since there are five 0's, zero 1's, and two 2's. 

We denote the cyclic rotation of the word $s=s_1s_2\ldots s_n$ starting with $s_i$ by $$s_{[i,\rightarrow]} = s_is_{i+1}\ldots s_ns_1\ldots s_{i-1}.$$ For example, if $s=010111$, then $s_{[4,\rightarrow]} = 111010$ and $s_{[6,\rightarrow]}=101011.$ We denote the set of cyclic rotations of $s$ by $[s]$, so that $$[s]=\{s_{[i,\rightarrow]}: i \in [n]\}.$$ For example, $$[01011]=\{01011,10110, 01101,11010, 10101\}.$$ Notice that the relation given by ``$s\sim t$ if $s\in[t]$'' is an equivalence relation. This equivalence relation preserves periodicity. That is, if $s\sim t$ and $s$ is $r$-periodic, then $t$ is $r$-periodic. In particular, if $s$ is primitive, then so is $t$. We refer to a single equivalence class $[s]$ as a \textit{necklace}.  

We refer to the set $\{[s]: s\in \W_k(n)\}$ as the set of \textit{necklaces} of length $n$ on $k$ letters and denote it by $\overline{\W}_k(n)$. A necklace $[s]$ is called \textit{$r$-periodic} (respectively, \textit{primitive}) if a representative $s$ of the equivalence class is $r$-periodic (respectively, primitive). Let $\N_k(n)$ denote the set of primitive necklaces on $k$ letters of length $n$, and let $L_k(n)$ denote the number of such necklaces. (This notation is standard as $L_k(n)$ is the number of \textit{Lyndon words}, where a Lyndon word refers to a specific representative of a given necklace.)
The following proposition is well-known and is easily shown using M\"obius inversion. We let $\mu$ denote the number-theoretic M\"obius function.
\begin{prop}
If $n\geq 1$, the number of necklaces of length $n$ on $k$ letters is equal to
\[ L_k(n) = \frac{1}{n}\sum_{d|n} \mu(d)k^{\frac{n}{d}}.\]
\end{prop}

For a given $\sigma \in \{+,-\}^k$ and word $s \in \W_k(n)$, we define $o_\sigma(s)$ to be $$o_\sigma(s)=|\{i \in [n] : s_i \in T^-_\sigma\}|.$$ That is, $o_\sigma(s)$ is the number of terms in $s$ that are locations of $-$ in $\sigma$. We define $o_\sigma([s])=o_{\sigma}(s)$ for any necklace $[s]$ as well. 
\begin{example}
If $\sigma = +-$, then $T^-_\sigma=\{1\}$, so $o_\sigma(s)$ is the number of 1's that appear in $s$. Thus, we have $o_\sigma(01001) = 2$ and $o_\sigma(01111) = 4$. 
\end{example}
\begin{example}
If $\sigma = -+--$, then $T_\sigma^-=\{0,2,3\}$, so $o_\sigma(s)$ is the number of 0's, 2's, or 3's that appear in $s$. Thus, $o_\sigma(1203123) = 5$ and $o_\sigma(0000103) = 6$. 
\end{example}
Let  $L^o_k(n;\sigma)$ denote the size of the set of primitive necklaces $[s]$ so that $s$ is a primitive word of size $n$ for which $o_\sigma(s)$ is odd.
%\todo{maybe change that notation? add a superscript o? probably need to do $\N^*$. Other notational issues?}
We define $\N^*_k(n;\sigma)\subseteq \overline{\W}_k(n)$ to be the set of primitive necklaces together with all 2-periodic necklaces $[s] = [q^2]$ such that $q$ is a primitive word of length $\frac{n}{2}$ for which $o_\sigma(q)$ is odd. Notice that when $n$ is odd these necklaces are exactly those from $\N_k(n)$ and when $n$ is even, the size of this set is exactly $L_k(n)+L^o_k(\frac n2;\sigma)$. For ease of notation, we set $L_k^*(n;\sigma) := L_k(n)+L^o_k(\frac n2;\sigma)$ with the assumption that $L^o_k\left(\frac n2 ;\sigma\right)=0$ when $n$ is odd. 
\begin{example}
Suppose $\sigma=+-$ and $n=8$. Then $\N_2^*(8;+-)$ is comprised of all the primitive necklaces given by the set $\N_2(8)$ together with the 2-periodic necklaces $[00010001]$ and $[01110111]$. 
\end{example}

The set $\N^*_k(n;\sigma)$ appears in many proofs in this paper and is a major part of the statement and proof of our main theorem, Theorem \ref{BIJECTION BETWEEN WORDS AND SIGSEG}. 
Using M\"obius inversion, it is a straightforward exercise to compute $L^o_k(n;\sigma)$ and $L^*_k(n;\sigma)$ for a given $\sigma$ of length 2 or 3 (as in the proof of Theorem~\ref{thm sigma3 equiv sigma4}). Therefore, in this paper we present the enumeration of $C_n^{\sigma}$ for each $\sigma\in\{+,-\}^3$ in terms of these numbers. 

\begin{lemma}\label{REVERSE WORD}
If $n\geq 1$, then $L_3^*(n; \sigma) = L_3^*(n; \sigma_r)$. 
\end{lemma}

\begin{proof}
It is enough to show that $L^o_3(\frac{n}{2}; \sigma) = L^o_3(\frac{n}{2}; \sigma_r)$. When $n$ is odd, the equality is clear, and thus we assume $n$ is even. We define the following bijection from necklaces $[s]$ of length $\frac{n}{2}$ with $o_\sigma(s)$ odd to necklaces $[s']$ of size $\frac{n}{2}$ with $o_{\sigma_r}(s')$ odd. Let $[s]$ be a primitive necklace of length $\frac{n}{2}$ for which $o_\sigma(s)$ is odd and take $[s']$ to be the necklace obtained by setting $s_i' = k-1-s_i$. Notice that $T_{\sigma_r}^- = \{k-1-i : i \in T_\sigma^-\}$, and thus $o_{\sigma_r}(s') = |\{i \in [\frac{n}{2}] : s'_i \in T^-_{\sigma_r} \}|$ is equal to $o_\sigma(s)$. Therefore, $o_{\sigma_r}(s')$ is odd and so $L^o_3(\frac{n}{2}; \sigma) = L^o_3(\frac{n}{2}; \sigma_r)$.
\end{proof}

\subsection{An ordering on words and necklaces}

Let $s= s_1s_2\ldots s_n$ and $t=t_1t_2\ldots t_n$ be words. Let $\prec_\sigma$ be the linear order on $\W_k(n)$ defined by $s\prec_\sigma t$ if one of the following holds:
\begin{enumerate}[(1)]
\item $s_1 < t_1$,
\item $s_1 = t_1 \in T^+_\sigma$ and $s_2s_3\ldots s_n\prec_\sigma t_2t_3\ldots t_n,$ or
\item $s_1 = t_1 \in T^-_\sigma$ and $t_2t_3\ldots t_n\prec_\sigma s_2s_3\ldots s_n.$
\end{enumerate}
Equivalently, $s\prec_\sigma t$ if, letting $j\ge1$ be the smallest such that $s_j \neq t_j$, either
$o_\sigma(s_{1}s_2\ldots s_{j-1})$ is even and $s_j<t_j$, or $o_\sigma(s_{1}s_2\ldots s_{j-1})$ is odd and $s_j>t_j$. 

\begin{example}
If $\sigma=+^k$, then $s\prec_\sigma t$ exactly when $s<t$ under the standard lexicographical ordering. However, if $\sigma=+--$, then $02212\prec_\sigma 02211$ since their first four terms agree, an odd number of these first four terms are elements of $T_\sigma^-$, and in the fifth position $2>1$. 
\end{example}

For $\sigma \in \{+,-\}^k$, let us define a map $\Pi_\sigma: \N^*_k(n;\sigma)\to\bar\S_n$. First, suppose $s=[s_1s_2\ldots s_n]$ is a primitive necklace. Then the words $s_{[i,\rightarrow]}$ are all distinct for $i\in[n]$. Thus, there exists a unique permutation $\tau\in\S_n$ such that $\tau_i<\tau_j$ if and only if $s_{[i,\rightarrow]} \prec_\sigma s_{[j,\rightarrow]}$. That is, $\tau$ records the ordering under $\prec_\sigma$ of the shifts of $s$. In this case, let $\Pi_\sigma([s]) = [\tau]$. If $[s]$ is not a primitive necklace, then $s = q^2$ for some primitive word $q$ of length $r:=\frac{n}{2}$ so that $o_\sigma(q)$ is odd. In this case, the $s_{[i,\rightarrow]}$ are all distinct for $1\leq i\leq\frac{n}{2}$ and $s_{[i,\rightarrow]}=s_{[i-r,\rightarrow]}$ for all $\frac{n}{2}<i \leq n$. % \todo{define how $\Pi$ works for these words. Might need a lemma.} 
%In the case that $s_{[i,\rightarrow]}=s_{[i-r,\rightarrow]}$, take $\pi_i<\pi_{i-r}$ exactly when $o_\sigma(s_1s_2\ldots s_{i-r})$ is odd and $\pi_i>\pi_{i-r}$ otherwise. \todo{make sure this works... do example and say why it works} 
For $[s]$ non-primitive, we define $\Pi_\sigma([s]) = [\tau]$, where for a given representative $s=s_1s_2\ldots s_n$, $\tau$ is the unique permutation such that:
\begin{itemize}
\item  $\tau_i<\tau_j$ when $s_{[i,\rightarrow]}\prec_\sigma s_{[j,\rightarrow]}$;
\item $\tau_1<\tau_{r+1}$; and 
\item for all $1\leq i <\frac{n}{2}$, 
\begin{itemize}
\item if $\tau_{i}<\tau_{r+i}$ and $s_i\in T_\sigma^+$, then $\tau_{i+1}<\tau_{r+i+1}$,
\item  if $\tau_{i}<\tau_{r+i}$ and $s_i\in T_\sigma^-$, then $\tau_{i+1}>\tau_{r+i+1}$, 
\item  if $\tau_{i}>\tau_{r+i}$ and $s_i\in T_\sigma^+$, then $\tau_{i+1}>\tau_{r+i+1}$, and 
\item  if $\tau_{i}>\tau_{r+i}$ and $s_i\in T_\sigma^-$, then $\tau_{i+1}<\tau_{r+i+1}$.
\end{itemize}
\end{itemize}
Under this definition, though taking different representatives $s^{(1)}$ and $s^{(2)}$ of the necklace $[s]$ would result in different permutations $\tau^{(1)}$ and $\tau^{(2)}$, they would be the same up to cyclic rotation. That is, we would have that $[\tau^{(1)}]=[\tau^{(2)}]$, and thus $\Pi_\sigma$ is well-defined. First, let us see an example of $\Pi_\sigma$ for a primitive necklace. 
\begin{example}\label{ex: prim}
Suppose $\sigma=+-$ and $[s]=[0010011]$. Taking $s=0010011$, there are seven distinct words to compare under the ordering $\prec_\sigma$, namely
\begin{align*}
s_{[1,\rightarrow]}& = 0010011, \\
s_{[2,\rightarrow]}& = 0100110, \\
s_{[3,\rightarrow]}& = 1001100, \\
s_{[4,\rightarrow]}& = 0011001, \\
s_{[5,\rightarrow]}& = 0110010, \\
s_{[6,\rightarrow]}& = 1100100, \text{ and} \\
s_{[7,\rightarrow]}& = 1001001. 
\end{align*}
Under this ordering, we clearly have that $s_{[1,\rightarrow]}, s_{[2,\rightarrow]}, s_{[4,\rightarrow]}$ and $s_{[5,\rightarrow]}$ are the four smallest words under the ordering determined by $\prec_\sigma$ since in each case, their first element is 0. To compare each pairwise, we check up to the first place the two words disagree. If they disagree in the $i$-th position and there are an even number of ones preceding that, the word with 0 in the $i$-th position is smaller. If there are an odd number of ones, the word with 1 in the $i$-th place is smaller. 

For example,  $s_{[1,\rightarrow]}$ and  $s_{[4,\rightarrow]}$ disagree for the first time in the 4-th position. The first three elements, 001, have an odd number of ones, so $s_{[4,\rightarrow]}$ is smaller than $s_{[1,\rightarrow]}$. %Since $0\in T_\sigma^+$, we can shift these words by one an see that $s_{[1,\rightarrow]}$ and $s_{[4,\rightarrow]}$ are the two smallest words. 
Continuing this process, we obtain that $\tau_1=2$, $\tau_2=4$, $\tau_4=1$, and $\tau_5=3$. 
Similarly, we can compare $ s_{[3,\rightarrow]}, s_{[6,\rightarrow]}$ and $s_{[7,\rightarrow]}$ and find that $\tau_3=7$, $\tau_6=5$, and $\tau_7=6$. 
Thus $\Pi_\sigma([0010011]) = [2471356]$. 
\end{example}
Now, let us consider a non-primitive necklace as an example.
\begin{example}\label{ex: not prim}
 Suppose that $\sigma = +--$ and $[s]=[01210121]$. Let us take the representative $s=01210121$ and compute the corresponding $\tau$. Clearly, we can see that 
$$
s_{[1,\rightarrow]}= s_{[5,\rightarrow]} \prec_\sigma s_{[2,\rightarrow]}= s_{[6,\rightarrow]} \prec_\sigma s_{[4,\rightarrow]}= s_{[8,\rightarrow]} \prec_\sigma s_{[3,\rightarrow]}=s_{[7,\rightarrow]}.
$$ 
Now, $\tau_i<\tau_j$ if $s_{[i,\rightarrow]}\prec_\sigma s_{[j,\rightarrow]}$. When $j=i+4$, we must decide what occurs. We take $\tau_1<\tau_5$. Since $s_1=0\in T_\sigma^+$, we also have $\tau_2<\tau_6$. Since $s_2=1\in T_\sigma^-$, we have $\tau_3>\tau_7$. Since $s_3=2\in T_\sigma^-$, we obtain $\tau_4<\tau_8$. Taken together, we get $[\tau]=[13852476]$.
\end{example}

\section{Main Theorem}\label{section main theorem}
This section includes the main theorem in this paper, Theorem~\ref{BIJECTION BETWEEN WORDS AND SIGSEG}, which states that, for $\sigma \in \{+,-\}^k$, the set of $\sigma$-segmentations of permutations in $\mathcal C_n^\sigma$ is equinumerous to $\N^*_k(n;\sigma)$. We will use Theorem~\ref{BIJECTION BETWEEN WORDS AND SIGSEG} to enumerate the set of cyclic permutations in several vector grid classes in Sections \ref{section 2-vector} and \ref{section 3-vector}.  

\begin{theorem}\label{BIJECTION BETWEEN WORDS AND SIGSEG}
For $k \geq 2$, $n \geq3$ and $\sigma \in \{+,-\}^k$,
	\[L_k^*(n;\sigma)= \sum_{\pi \in \mathcal C_n^\sigma} |\{\e: \e \text{ is a $\sigma$-segmentation of $\pi$}\}|.\]
Furthermore, the number of necklaces in $\N^*_k(n;\sigma)$ with evaluation $\textbf{a}=(a_0, a_1,\ldots, a_{k-1})$ is equal to the number of permutations $\pi \in\C_n^\sigma$ that admit the $\sigma$-segmentation $\e=(e_0, e_1,\ldots, e_k),$ where $e_0=0$ and $e_i = \displaystyle\sum_{j<i} a_j$ for $i\in[k]$.
\end{theorem}

%\todo{discussion of how to use this theorem. }

For any $\sigma\in\{+,-\}^k$ and any $n\geq 1$, let us define the set $\SS(n;\sigma)$ as $$\SS(n;\sigma)=\{(\pi,\mathbf{e}) : \pi\in\C_n^\sigma, \ \mathbf{e} \text{ is a $\sigma$-segmentation of $\pi$}\}.$$
That is, $\SS(n;\sigma)$ consists of all cyclic \textit{gridded permutations} \cite{AAVRB,BEVAN2014,BEVAN2015}.  To prove Theorem \ref{BIJECTION BETWEEN WORDS AND SIGSEG}, we construct a bijection from $\N^*_k(n;\sigma)$ to $\SS(n;\sigma)$ that sends a necklace with evaluation $\textbf{a} =(a_0, a_1,\ldots, a_{k-1})$ to some $(\pi, \e)$, where $\pi \in\C_n^\sigma$ admits a $\sigma$-segmentation $\e=(e_0, e_1,\ldots, e_k)$ with $e_0=0$ and $e_i = \sum_{j<i} a_j$ for $i\in[k]$. This bijection  associates a permutation to the necklace by observing the relative order of the elements that comprise the necklace with respect to a certain ordering determined by $\sigma$. This proof requires several lemmas and a few additional definitions.

%\todo{we should probably call it $\varphi_\sigma$}
%\todo{here, write what happens in the case $s=q^2$ and do examples}

Let $\mathbb{N}_0$ denote the set of nonnegative integers and let $\mathbb{N}_0^k$ denote the set of sequences with length $k$ and entries from $\mathbb{N}_0$. Suppose $\sigma\in\{+,-\}^k$ and let $\varphi_\sigma:\N^*_k(n;\sigma) \to \S_n\times \mathbb{N}_0^k$ be defined by $\varphi_\sigma([s]) = (\pi,\e)$ where $\pi= \bar\theta \circ \Pi_\sigma([s])$ and given that the evaluation of $[s]$ is $\textbf{a}=(a_0, a_1,\ldots, a_{k-1})$, we have that $\e=(e_0, e_1, \ldots, e_k)$ with $e_0=0$ and $e_i = \sum_{j<i} a_j$ for $i\in[k]$. We will show in Lemmas \ref{lem: ppat} through \ref{lem: inverses} that $\varphi_\sigma([s])$ is indeed an element of $\SS(n;\sigma)$ and that the map $\varphi_\sigma$ is a bijection. First, consider the following two examples that demonstrate that $\varphi_\sigma([s]) \in \SS(n;\sigma)$.

\begin{example}
Suppose $\sigma=+-$ and consider the necklace $[s]=[0010011]$. Notice that $[s]$ has evaluation $\mathbf{a}=(4,3)$ since it has four 0's and three 1's. In Example \ref{ex: prim}, we saw that $\Pi_\sigma([0010011]) = [2471356].$ Thus, $\varphi_\sigma([0010011]) = (\pi, \e)$, where $\pi = (2471356) = 3457621$ and $\e=(0,4,7)$. Notice that $3457621 \in \C^\sigma_7$ since it is increasing, then decreasing, and that $\e$ is a $\sigma$-segmentation of $\pi$. 
\end{example}
\begin{example}
Suppose $\sigma = +--$ and $[s]=[01210121]$, which has evaluation $\mathbf{a}=(2,4,2)$. We saw in Example \ref{ex: not prim} that $\Pi_\sigma([01210121]) = [13852476]$.  Thus, $\varphi_\sigma([01210121]) = (\pi, \e)$, where $\pi = (13852476) =34872165$ and $\e=(0,2, 6,8)$. Notice that $34872165 \in \C^\sigma_8$ since 34 is increasing, 8721 is decreasing, and 65 is decreasing, and that $\e=(0,2,6,8)$ is a $\sigma$-segmentation of $\pi$.
\end{example}
In both examples above, we have that $\varphi_\sigma([s]) \in \SS(n;\sigma)$. We will see in Lemma \ref{lem: ppat} that in general this will happen. 

Given a permutation $\tau$ of length $n$ and a vector $\e=(e_0,e_1,\ldots,e_k)$ with $0=e_0\leq e_1\leq \cdots \leq e_k=n$, let us define the \textit{$\tau$-monotone word $s=s_1s_2\ldots s_n$ induced by $\e$} by letting $s_i = t$ whenever $e_{t}< \tau_i\leq e_{t+1}$. For example, if $\tau=34872165$ and $\e=(0,2,6,8)$, then $s=11220011$ is the $\tau$-monotone word induced by $\e$.  Notice that if the evaluation of $s$ is $\mathbf{a}=(a_0, a_1, \ldots, a_{k-1})$, then $e_i=\sum_{j<i} a_j$ for all $i\in[k]$. We say that $[s]$ is the $[\tau]$-monotone necklace if $s$ is a $\tau$-monotone word for some representatives $s$ and $\tau$ of $[s]$ and $[\tau]$, respectively. 

Let $\sigma\in\{+,-\}^k$ be arbitrary and let $\psi_\sigma:\SS(n;\sigma) \to \overline{\W}_k(n)$ be defined by $\psi_\sigma((\pi, \e)) = [s]$, where $[s]=[s_1s_2\ldots s_n]$ is the $[\tau]$-monotone necklace for $[\tau]=\bar\theta^{-1}(\pi)$.
We will show that $\varphi_\sigma$ gives a bijection from $\N^*_k(n;\sigma)$ onto $\SS(n;\sigma)$ and that $\psi_\sigma$ is the inverse of $\varphi_\sigma$. 

\begin{example}
Suppose $\sigma = ++-$. Let us consider the cyclic permutation $\pi = 46723581= (14265378)\in \C_8^\sigma$ together with the $\sigma$-segmentation $\e=(0,3,7,8)$. Then $\psi_\sigma((\pi,\e))$ will have three 0's, four 1's, and one 2. The cycle notation of $\pi$ informs us where to place each letter. Thus, we obtain $\psi_\sigma((\pi,\e))=[01011012]$. Notice that $\varphi_\sigma([01011012]) = (\pi,\e)$. If we consider instead the $\sigma$-segmentation $\mathbf{f}=(0,3,6,8)$ of $\pi$, then we obtain $\psi_\sigma((\pi,\mathbf{f}))=[01011022]$. Again, we get that $\varphi_\sigma([01011022]) = (\pi,\mathbf{f})$.
\end{example}

%In part (c) of this first lemma, we show that the image of $\varphi_\sigma$ is contained in $\SS(n;\sigma)$.
\begin{lemma} \label{lem: ppat}
Suppose $n,k\geq 1$ and $\sigma\in\{+,-\}^k$. Let $[s]\in \N^*_k(n;\sigma)$. Then $\varphi_\sigma([s]) \in \SS(n;\sigma)$. 
%Let $\sigma\in\{+,-\}^k$ be arbitrary, let $[s]=[s_1s_2\ldots s_n]\in \N_k(n;\sigma)$ with evaluation  $\textbf{a}=(a_0, a_1,\ldots, a_{k-1})$, and let $[\tau]\in\bar\S_n$ such that $[\tau]=\Pi_\sigma([s])$. For $1\le t\le k$, let
%$$e_i = \sum_{j<i} a_j$$ and let $e_0=0$. The following statements hold:
%\begin{enumerate}[(a)]
%\item For $0\leq t<k$, $s_i=t$ if and only if $e_{t} < \tau_i\le e_{t+1}$.
%\item If $e_{t} < \tau_i<\tau_j\leq e_{t+1}$, then $\tau_{i+1}<\tau_{j+1}$ if $t \in T^+_\sigma$, and $\tau_{i+1}>\tau_{j+1}$ if $t \in T^-_\sigma$,
%where we let $\tau_{n+1}:=\tau_1$.
%\item The sequence $\e=(e_0,e_1,\dots,e_k)$ is a $\sigma$-segmentation of $\pi:=\theta(\tau)$. In particular, $(\pi,\e) \in \SS(n;\sigma)$.
%\end{enumerate}
\end{lemma}

\begin{proof}
Let $\Pi_\sigma(s)=\tau$ and $\varphi_\sigma(s)=(\pi,\e)$, where $\e=(e_0,e_1,\ldots,e_k)$.  It is clear from the definition of $\Pi_\sigma$ that for all $a,b\in[n]$, $\tau_a<\tau_b$ implies $s_a\le s_b$. It follows that for $0\leq t<k$, $s_i=t$ if and only if $e_{t} < \tau_i\le e_{t+1}$. %
Now, suppose that $e_{t} < \tau_i<\tau_j\leq e_{t+1}$, and so $s_{[i, \rightarrow]} \preceq_\sigma s_{[j,\rightarrow]}$ with $s_i = s_j=t$.
If $t \in T^+_\sigma$, then $s_{[i+1, \rightarrow]} \prec_\sigma s_{[j+1,\rightarrow]}$, and so $\tau_{i+1}<\tau_{j+1}$.
Similarly, if $t \in T^-_\sigma$, then $s_{[j+1,\rightarrow]} \prec_\sigma s_{[i+1,\rightarrow]}$, and so $\tau_{i+1}>\tau_{j+1}$.

Now let $0\le t<k$, and suppose that the indices $j$ such that $s_j=t$ are $j_1,\dots, j_m$, ordered in such a way that $\tau_{j_1}< \tau_{j_2}<\dots< \tau_{j_m}$,
where $m=e_{t+1}-e_{t}$. Then $\tau_{j_\ell}=e_{t-1}+\ell$ for $1\le\ell\le m$,
and by the previous paragraph, we also have that $\tau_{j_1+1}< \tau_{j_2+1}<\dots< \tau_{j_m+1}$ if $t\in T^+_\sigma$,
and $\tau_{j_1+1}> \tau_{j_2+1}>\dots> \tau_{j_m+1}$ if $t\in T^-_\sigma$.
Using that $\tau_{j_\ell+1}=\pi_{\tau_{j_\ell}}=\pi_{e_{t}+\ell}$, this is equivalent to
$\pi_{e_{t}+1}< \pi_{e_{t}+2}<\dots<\pi_{e_{t}+m}$ if $t\in T^+_\sigma$,
and $\pi_{e_{t}+1}>\pi_{e_{t}+2}>\dots>\pi_{e_{t}+m}$ if $t\in T^-_\sigma$. Note that $e_{t}+m = e_{t+1}$, so this condition states that $\e$ is a $\sigma$-segmentation of $\pi$.
Since $\pi$ is a cyclic permutation, this proves that $\pi\in\C^\sigma$.
\end{proof}

%\todo{the next two lemmas will...}

\begin{lemma}\label{lem: Csigma}
Let $\sigma\in\{+,-\}^k$ be arbitrary, let $\tau\in\S_n$ with $\pi = \theta(\tau)\in \C^\sigma$, and suppose that
$\e=(e_0, e_1, \ldots, e_k)$ is a $\sigma$-segmentation of $\pi$.
Suppose that $e_{t} < \tau_i<\tau_j\leq e_{t+1}$ for some $1\leq i,j\leq n$. Then $\tau_{i+1}<\tau_{j+1}$ if $t \in T^+_\sigma$, and $\tau_{i+1}>\tau_{j+1}$ if $t \in T^-_\sigma$,
where we let $\tau_{n+1}:=\tau_1$.
\end{lemma}

\begin{proof}
Since $e_{t} < \tau_i<\tau_j\leq e_{t+1}$, both $\pi_{\tau_i}$ and $\pi_{\tau_j}$ lie in the segment $\pi_{e_{t}+1}\dots \pi_{e_{t+1}}$. If $t \in T^+_\sigma$, this segment is increasing,
so $\tau_{i+1}=\pi_{\tau_i}<\pi_{\tau_j}=\tau_{j+1}$. The argument is analogous if $t \in T^-_\sigma$.
\end{proof}

\begin{lemma}\label{lem: psi}
Suppose $n,k\geq 1$ and $\sigma \in \{+,-\}^k$. Let $\pi \in \C_n^\sigma$ and $\e$ be a $\sigma$-segmentation of $\pi$. Then $\psi_\sigma((\pi,\e)) \in \N^*_k(n;\sigma)$. 
\end{lemma}
\begin{proof}
Let $\psi_\sigma((\pi,\e))=[s]$ and $\tau=\theta^{-1}(\pi)$.  Assume that $\e=(e_0,e_1,\ldots,e_k)$.  If $s$ is primitive, we are done. Suppose that $s$ is not primitive, so it can be written as $q^m$ for some $m\ge 2$ and some primitive word $q$ with $|q|=r=\frac{n}{m}$.
Then, $s_i = s_{i+r}$ for all $i$ (using addition modulo $n$ in the index for the remainder of the proof). %Let $g = |\{i\in[r] :  s_i \in T^-_\sigma\}|$.
Fix $i$, and let $t=s_i = s_{i+r}$. Because of the way that $s$ is defined, we must have $e_t< \tau_i,\tau_{i+r}\leq e_{t+1}$, so we can apply Lemma~\ref{lem: Csigma} to this pair.

Suppose first that $g=o_\sigma(q)$ is even.
If $\tau_i<\tau_{i+r}$, then applying Lemma~\ref{lem: Csigma} $r$ times we get $\tau_{i+r}<\tau_{i+2r}$,
since the inequality involving $\tau_{i+\ell}$ and $\tau_{i+r+\ell}$ switches exactly $g$ times as $\ell$ increases from $0$ to $r$.  Starting with $i=1$ and applying this argument repeatedly, we see that if $\tau_1<\tau_{1+r}$, then
$$\tau_1<\tau_{1+r}<\tau_{1+2r}<\dots <\tau_{1+(m-1)r}<\tau_{1+mr}=\tau_1,$$ which is a contradiction.
A symmetric argument shows that if $\tau_1>\tau_{1+r}$, then $$\tau_1>\tau_{1+r}>\tau_{1+2r}>\dots >\tau_{1+(m-1)r}>\tau_{1+mr}=\tau_1,$$ also a contradiction.

It remains to consider the case that $g=o_\sigma(q)$ is odd. If $m$ is even and $m\ge4$, then letting $q' = q^2$ we have $s = (q')^{\frac{m}{2}}$.
Letting $r'=|q'|=2r$ and $g'= o_\sigma(q')=2g$, the same argument as above using $r'$ and $g'$ yields a contradiction.
If $m$ is odd, suppose without loss of generality that $\tau_1<\tau_{1+r}$.
Note that applying Lemma~\ref{lem: Csigma} $r$ times to the inequality $\tau_i<\tau_{i+r}$ (respectively $\tau_i>\tau_{i+r}$) yields $\tau_{i+r}>\tau_{i+2r}$ (respectively $\tau_{i+r}<\tau_{i+2r}$) in this case, since the inequality involving $\tau_{i+\ell}$ and $\tau_{i+r+\ell}$ switches an odd number of times.
Consider two cases:
\begin{itemize}
\item If $\tau_{1}<\tau_{1+2r}$, then Lemma~\ref{lem: Csigma} applied repeatedly in blocks of $2r$ times yields $\tau_1<\tau_{1+2r}<\tau_{1+4r}<\dots< \tau_{1+(m-1)r}$.
Applying now Lemma~\ref{lem: Csigma} $r$ times starting with $\tau_1<\tau_{1+(m-1)r}$ gives $\tau_{1+r}>\tau_{1+mr}=\tau_1$, which contradicts the assumption $\tau_1<\tau_{1+r}$.

\item If $\tau_{1}>\tau_{1+2r}$, applying Lemma~\ref{lem: Csigma} $r$ times we get $\tau_{1+r}<\tau_{1+3r}$, and by repeated application of the lemma in blocks of $2r$ times it follows that
$\tau_{1+r}<\tau_{1+3r}<\tau_{1+5r}<\dots< \tau_{1+(m-2)r}<\tau_{1+mr}=\tau_1$, contradicting again the assumption $\tau_1<\tau_{1+r}$.
\end{itemize}

The only case left is when $g$ is odd and $m = 2$, that is, when $s_1s_2\dots s_n = q^2$ and $q$ has an odd number of letters in $T^-_\sigma$. Thus, we have shown that $[s]\in \N^*_k(n;\sigma)$.
\end{proof}

\begin{lemma}\label{lem: inverses}
Suppose $k\geq 1$ and $\sigma \in \{+,-\}^k$.  Then $\varphi_\sigma$ is a bijection from $\N^*_k(n;\sigma)$ onto $\SS(n;\sigma)$, and its inverse is $\psi_\sigma$. 
\end{lemma}

%\begin{lemma}\label{lem: s1sn}
%Let $\sigma\in\{+,-\}^k$ be arbitrary, let $\pi\in\C^\sigma_n$, and let $\tau\in\S_n$ so that $\theta(\tau)=\pi$. Take any $\sigma$-segmentation of $\pi$ and let $s=s_1s_2\ldots s_n$ be the $\tau$-monotone word induced by it. Then $[s]\in \N_k(n;\sigma)$ and $\Pi_\sigma([s])=\tau$. 
%\end{lemma}
\begin{proof}
Fix $n\geq 1$. We will first show that for a given $\pi\in \C_n^\sigma$ and a $\sigma$-segmentation $\e$ of $\pi$, we have that $(\pi,\e) = \varphi_\sigma(\psi_\sigma((\pi,\e)))$. Let $\tau\in\S_n$ so that $\theta(\tau)=\pi$ and let $s=s_1s_2\ldots s_n$ be the $\tau$-monotone word induced by the $\sigma$-segmentation $\e$. 
We need to show that $\Pi_\sigma([s])=[\tau]$. 

Suppose that $\tau_i<\tau_j$. If $s_{[i,\rightarrow]}\neq s_{[j,\rightarrow]}$, let $a\ge0$ be the smallest such that $s_{i+a}\neq  s_{j+a}$,
and let $h=|\{0\le \ell \le a-1 : s_{i+\ell} \in T^-_\sigma\}|$.
 If $h$ is even, then Lemma \ref{lem: Csigma} applied $a$ times shows that $\tau_{i+a}<\tau_{j+a}$. Since $s_{i+a}\neq s_{j+a}$, we must then have $s_{i+a} < s_{j+a}$ because of $\tau$-monotonicity.
Thus, $s_{[i,\rightarrow]}\prec_\sigma s_{[j,\rightarrow]}$ by definition of $\prec_\sigma$, since the word $s_is_{i+1}\dots s_{i+a-1}= s_js_{j+1}\dots s_{j+a-1}$ has an even number of letters in $T^-_\sigma$.
 Similarly, if $h$ is odd, then Lemma \ref{lem: Csigma} shows that  $\tau_{i+a}>\tau_{j+a}$. Since $s_{i+a}\neq s_{j+a}$, we must have $s_{i+a} > s_{j+a}$, and thus $s_{[i,\rightarrow]}\prec_\sigma s_{[j,\rightarrow]}$
by definition of $\prec_\sigma$.
If $s$ is primitive, the case $s_{[i,\rightarrow]}=s_{[j,\rightarrow]}$ can never occur when $i\neq j$, and so $\pi_i<\pi_j$ if and only if $s_{[i,\rightarrow]} \prec_\sigma s_{[j,\rightarrow]}$. 
It follows that $\Pi_\sigma([s])=[\tau]$.

 Let $r:=\frac{n}{2}$.
 If $s$ is not primitive, we can only have $s_{[i,\rightarrow]}=s_{[j,\rightarrow]}$ if $i = j\pm r$. Suppose $\rho=\Pi_\sigma([s])$. In this case, we have that $\rho_1<\rho_{1+r}$, and for $1<i\leq r$, $\rho_i<\rho_{i+r}$ exactly when either $\rho_{i-1}<\rho_{i+r-1}$ with $s_{i-1}\in T^+_\sigma$ or  $\rho_{i-1}>\rho_{i+r-1}$ with $s_{i-1}\in T^-_\sigma$. Thus, if $\tau_1<\tau_{1+r}$, we get $\rho=\tau$. If $\tau_1>\tau_{1+r}$, we get $\rho=\tau_{r+1}\ldots \tau_n\tau_1\ldots\tau_r$, a cyclic rotation of $\tau$. Either way, $[\rho]=[\tau]$ and so $\Pi_\sigma([s])=[\tau]$.  

Let $[s]\in\N^*_k(n;\sigma)$ with evaluation $\mathbf{a}=(a_0, a_1, \ldots, a_{k-1})$. We will see that $\psi_\sigma(\varphi_\sigma([s]))=[s]$. Let $[\tau]=\Pi_\sigma([s])$ and let $\e=(e_0,e_1,\ldots,e_k)$ be the partial sums given by $e_0=0$ and $e_i=\sum_{j<i}a_j$ for $i\in[k]$. It is enough to show that the $[\tau]$-monotone necklace induced by $\e$ is $[s]$. However, this is clear from the first paragraph of the proof of Lemma \ref{lem: ppat}.
\end{proof}

\begin{proof}[Proof of Theorem \ref{BIJECTION BETWEEN WORDS AND SIGSEG}]
Since the left hand side of the formula is the size of the set $\N^*_k(n;\sigma)$ and the right hand side is the size of the set $\SS(n;\sigma)$, Lemma \ref{lem: inverses} implies the equality. Since $\varphi_\sigma$ sends necklaces with evaluation $\mathbf{a}=(a_0, a_1,\ldots, a_{k-1})$ to permutations $\pi \in\C_n^\sigma$ and the associated $\sigma$-segmentation $\e=(e_0, e_1, \ldots, e_k)$ with $e_0=0$ and $e_i = \displaystyle\sum_{j<i} a_j$ for $i\in[k]$, we are done.
%We will show that $\varphi_\sigma$ is a bijection between $\N_k(n;\sigma)$ and $\SS(n;\sigma)$. 
 %Lemma \ref{lem: ppat} tells us that the image of $\varphi_\sigma$ is indeed $\SS(n;\sigma)$.
%Given a $\sigma$-segmentation $\e$ of some permutation in $\pi \in \C_n^\sigma$ and a permutation $\tau\in\S_n$ so that $\theta(\tau)=\pi$, we know from Lemma \ref{lem: s1sn} that the $\tau$-monotone word $s$ induced by it is either primitive or of the form $q^2$ where $g = |\{i\in[r] :  s_i \in T^-_\sigma\}|$ is odd, and thus $[s]\in \N_k(n;\sigma)$. We also know from this lemma that $\Pi_\sigma([s])=\tau$ and thus $\varphi_\sigma([s])=\pi$. Thus $\varphi_\sigma$ is invertible.% When $s$ is $n$-periodic, we can invert this map by taking the pattern of $s$ to obtain a periodic pattern (with some $\sigma$-segmentation induced by the word $s_1s_2\ldots s_n$). It remains to show that we can similarly define the `periodic pattern' of length $n$ of an $r$-periodic word with $g$ odd. The proof is an extension of the proof of Theorem \ref{thm:enumrs2mod4}, where this is done for the case when $\sigma = -^k$. 
\end{proof}

\section{Cycles in length 2 vector grid classes}\label{section 2-vector}

Here, we provide the enumeration of $\C_n^{++}$, $\C_n^{+-}$, $\C_n^{-+}$, and $\C_n^{--}$ as corollaries of Theorem \ref{BIJECTION BETWEEN WORDS AND SIGSEG}. These theorems also appear in \cite{gessel, ARCHERELIZALDE2014}. 
We also enumerate the number of permutations in $\C_n^{+-}$ with a peak at $i$, which will be useful for enumerating $\C_n^\sigma$ for $\sigma\in\{+,-\}^3$ in Section \ref{section 3-vector}.

\subsection{Enumerating $\C_n^{++}$, $\C_n^{+-}$, $\C_n^{-+}$, and $\C_n^{--}$}

\begin{theorem}\label{thm: ++ and --}
If $n\geq 2$, then $|\C_n^{++}| = L_2(n)$ and $|\C_n^{--}| = L_2^*(n;--).$
\end{theorem}

\begin{proof}
First suppose $\sigma=++$ and $\pi\in\C_n^{\sigma}$. Then $\pi$ is a cyclic permutation with at most one descent. Since $\pi$ is cyclic and $n\geq 2$, $\pi$ must have exactly one descent. Suppose the descent of $\pi$ occurs at position $i$. Then there is only one $\sigma$-segmentation of $\pi$, namely $(0,i,n)$. Thus $|\C_n^{++}| = L^*_2(n;\sigma)$. Since $T_\sigma^-=\varnothing$ in this case, we obtain $|\C_n^{++}| = L_2(n)$. 

Similarly, if $\pi\in\C_n^{--}$, then $\pi$ has exactly one ascent and thus one $\sigma$-segmentation. Thus the result follows.
\end{proof}

%\todo{define peak earlier}
\begin{theorem}
If $n\geq 2$, then $$|\C_n^{+-}| =|\C_n^{-+}| = \frac{L_2^*(n;+-)}{2}.$$
\end{theorem}
\begin{proof}
By Proposition~\ref{REVERSE SAME}, it is enough to consider $|\C_n^{+-}|$ only. If $\pi \in \C_n^{+-}$ and $\pi$ has a peak at position $i$, then there are two $\sigma$-segmentations of $\pi$, namely $(0, i-1, n)$ and $(0, i, n)$. Since this is true for all $\pi\in\C_n^{+-}$, the result follows.
\end{proof}

\subsection{Unimodal cycles with a given peak}

In this section, we enumerate the number of cyclic unimodal permutations (i.e., those permutations in $\C_n^{+-}$) with its peak in a given position.
We let $\bar\sigma = +-$ and denote by $\Lambda(n)$ the size of the set $\C_n^{\bar\sigma}$. %number of cyclic permutations in $\mathcal C_n^{\bar\sigma}$ (i.e., cyclic unimodal permutations) which was previously established in \cite{ARCHERELIZALDE2014} as
%	\[\Lambda(n)= L_2(n;\bar\sigma)=\frac{L_2^*(n;\bar\sigma)}{2}.\]  

Recall that we say there is a peak at $i$ if $\pi_i>\pi_{i-1}$ and $\pi_i>\pi_{i+1}$, or if $i=1$ and $\pi_1>\pi_2$, or if $i=n$ and $\pi_{n-1}<\pi_n$. 
Let $\Lambda(n,i)$ denote the number of cyclic permutations in $\C_n^{\bar\sigma}$ with a peak at position $i$ (i.e., the number of permutations $\pi\in \C^{\bar\sigma}_n$ with $\pi_i = n$). 
In order to enumerate the cyclic permutations in the $++-$- and $+--$-classes in Section~\ref{section 3-vector}, we will find $\Lambda(n,i)$, whose formula is given in Lemma~\ref{UP DOWN PEAK}.

Let $L_2(n,i)$ denote the number binary Lyndon words of length $n$ with $i$ 1's,  whose formula is well-known. Using M\"obius inversion, one can find that $$L_2(n,i) = \frac{1}{n} \sum_{d\mid\gcd(n,i)} \mu(d) {{n/d}\choose{i/d}}.$$

\begin{lemma}\label{WORD SYMMETRY}
If $n\geq 1$ and $0\leq i \leq n$, then $L_2(n,i)=L_2(n,n-i).$
\end{lemma}

\begin{proof}
By interchanging the 0's and 1's in each word, we obtain the result.
\end{proof}

\begin{lemma}\label{SUM PEAKS}
If $n\geq 1$ and $i \in [n-1]$, then
	\[\Lambda(n,i)+\Lambda(n,i+1)= L_2(n,i) \]
except in the case when $n$ is even and $n+i \equiv 2 \!\pmod 4;$ in this case,
	  \[\Lambda(n,i)+\Lambda(n,i+1)= L_2(n,i)+L_2\Big(\frac n2,\frac i2\Big). \]
\end{lemma}

\begin{proof}
Consider the second statement in Theorem~\ref{BIJECTION BETWEEN WORDS AND SIGSEG} for $\bar\sigma = +-$. It implies that the number of necklaces in $\N^*_2(n;\bar\sigma)$ with $i$ zeros and $n-i$ ones is equal to the number of cyclic permutations in $\C_n^{\bar\sigma}$ that admit the $\bar\sigma$-segmentation $(0,i,n)$. These are exactly the permutations with a peak at $i$ or $i+1$. In the case when either $n$ is odd or $n$ is even and $n+i \not\equiv 2\!\pmod 4$, the number of necklaces in $\N^*_2(n;\bar\sigma)$ with $i$ zeros and $n-i$ ones is exactly $L_2(n,n-i)$. When $n$ is even and $n+i\equiv 2 \!\pmod 4$, the number of such necklaces is $L_2(n,n-i) + L_2(\frac{n}{2}, \frac{n-i}{2})$. Since $L_2(n,n-i) = L_2(n,i)$ and $L_2(\frac{n}{2}, \frac{n-i}{2}) = L_2(\frac{n}{2}, \frac{i}{2})$ by Lemma~\ref{WORD SYMMETRY}, the result follows.
\end{proof}

\begin{lemma}\label{SYMMETRY LEMMA}
If $n\geq 1$ and $n \not\equiv 2\!\pmod 4$, then for $i \in [n]$, $$\Lambda(n,i)=\Lambda(n,n-i+1).$$ 
\end{lemma}
\begin{proof} The cases when $n\leq 2$ are easily checked. 
Now let $n>2$. If $i=1$, then $\Lambda(n,1)=0=\Lambda(n,n)$. We proceed by induction on $i$.  Assume $\Lambda(n,i-1)=\Lambda(n,n-i+2)$ for some $i \in\{2,3,\ldots,n-1\}$.  First suppose $n$ is odd. Since $L_2(n,i-1)=L_2(n,n-i+1)$, applying Lemma~\ref{SUM PEAKS} twice and Lemma~\ref{WORD SYMMETRY} we obtain
	\begin{eqnarray*} \Lambda(n,i-1)+\Lambda(n,i)&=&L_2(n,i-1) \\
		&=& L_2(n,n-i+1)\\
		&=&\Lambda(n,n-i+1)+\Lambda(n,n-i+2).
	\end{eqnarray*}
The result now follows from our inductive hypothesis.

Now suppose $n \equiv 0 \!\pmod 4$. If $i-1 \not\equiv 2\!\pmod 4$, then the above argument can be made by again using Lemma~\ref{SUM PEAKS} twice and Lemma~\ref{WORD SYMMETRY} once. If $i-1 \equiv 2\!\pmod 4$, notice that $n-i+1 \equiv 2 \!\pmod 4$ as well. Therefore, using the same lemmas, 
\begin{align*}\Lambda(n,i-1)+\Lambda(n,i)&=L_2(n,i-1) + L_2\bigg(\frac{n}{2}, \frac{i-1}{2}\bigg) \\
&=L_2(n,n-i+1) + L\bigg(\frac{n}{2}, \frac{n-i+1}{2}\bigg) \\
&=\Lambda(n,n-i+1)+\Lambda(n,n-i+2).\end{align*}
The result follows from our inductive hypothesis.
\end{proof}

\begin{lemma}\label{UP DOWN PEAK}
Let $n \geq 3$ and $i\in[n]$. If $n$ is odd, then
	\[\Lambda(n,i)= \sum_{j=1}^{i-1} (-1)^{i+j+1} L_2(n,j),\]
	and if $n$ is even, then
\[	\Lambda(n,i)=	 \sum_{j=1}^{i-1} (-1)^{i+j+1} L_2(n,j) + (-1)^{i+1}\!\!\sum_{\substack{k<i\\4|(n+k+2)}}\! L_2 \bigg(\frac n2,\frac k2\bigg).\]
\end{lemma}

\begin{proof}
Observe that $\Lambda(n,i)$ is only defined for $i \in [n]$.  Since $\Lambda(n,1)=0$ the result is clear for $i =1$. Thus, we assume $i >1$ and induct on $i$ with $i \in \{2,3, \ldots, n-1\}$.  If $i=2$, then from Lemma~\ref{SUM PEAKS} we obtain the equality
	\[\Lambda(n,1)+\Lambda(n,2)=L_2(n,1)\]
and the result holds.  Thus, we assume the result holds true for all permissible integers less than $i$.  First suppose $n$ is an odd integer.  It follows that 
	\[\Lambda(n,i-1)+\Lambda(n,i)=L_2(n,i-1)\]
by Lemma~\ref{SUM PEAKS}; therefore 
	\begin{eqnarray*}
		\Lambda(n,i) & = & L_2(n,i-1)-\Lambda(n,i-1)\\
			& = & L_2(n,i-1)-\sum_{j=1}^{i-2}(-1)^{i+j} L_2(n,j)\\ 
			& = & L_2(n,i-1)+\sum_{j=1}^{i-2}(-1)^{i+j+1} L_2(n,j)\\ 
			& = & \sum_{j=1}^{i-1}(-1)^{i+j+1} L_2(n,j),\\ 
	\end{eqnarray*}
where the second equality holds by inductive hypothesis. 
% Now assume that $n$ is an even integer. If $i$ is in integer satisfying $n+(i-1) \not\equiv 2\!\pmod 4$, then 
%	\[\Lambda(n,i-1)+\Lambda(n,i)=L_2(n,i-1)\]
%by Lemma~\ref{SUM PEAKS}.
%Following a similar argument as above, 
%	\begin{eqnarray*}
%		\Lambda(n,i) & = & L_2(n,i-1)-\Lambda(n,i-1)\\
%			& = & L_2(n,i-1)-\sum_{j=1}^{i-2}(-1)^{i+j} L_2(n,j) - (-1)^i\!\!\sum_{\substack{k<i-1\\4|(n+k+2)}}\! L_2 \bigg(\frac n2,\frac k2\bigg)\\ 
%			& = & L_2(n,i-1)+\sum_{j=1}^{i-2}(-1)^{i+j+1} L_2(n,j) + (-1)^{i+1}\!\!\sum_{\substack{k<i\\4|(n+k+2)}}\! L_2 \bigg(\frac n2,\frac k2\bigg)\\ 
%			& = & \sum_{j=1}^{i-1}(-1)^{i+j+1} L_2(n,j)+ (-1)^{i+1}\!\!\sum_{\substack{k<i\\4|(n+k+2)}}\! L_2 \bigg(\frac n2,\frac k2\bigg).\\ 
%	\end{eqnarray*}
%	Notice that we could rewrite the rightmost sum over $k<i$ since $4 \nmid (n+(i-1) +2)$. 

Now assume that $n$ is even and $i$ satisfies $n+i-1 \equiv 2 \!\pmod 4$. Then 
		\[\Lambda(n,i-1)+\Lambda(n,i)=L_2(n,i-1) + L_2\bigg(\frac{n}{2}, \frac{i}{2}\bigg)\]
by Lemma~\ref{SUM PEAKS}.
Following a similar argument as above, 

\vspace{6pt}

	\begin{tabular}{rcl}
		$\Lambda(n,i)$ & $=$ & $\ds L_2(n,i-1)-\Lambda(n,i-1) + L_2\bigg(\frac{n}{2}, \frac{i-1}{2}\bigg)$\\
			& $=$ & $\ds L_2(n,i-1)-\sum_{j=1}^{i-2}(-1)^{i+j} L_2(n,j)$ \\
			 & & $\ds - (-1)^i\!\!\sum_{\substack{k<i-1\\4|(n+k+2)}}\! L_2 \bigg(\frac n2,\frac k2\bigg)+ L_2\bigg(\frac{n}{2}, \frac{i-1}{2}\bigg)$\\ 
			& $=$ & $\ds L_2(n,i-1)+\sum_{j=1}^{i-2}(-1)^{i+j+1} L_2(n,j) $\\ 
			& & $\ds + (-1)^{i+1}\!\!\sum_{\substack{k<i\\4|(n+k+2)}}\! L_2 \bigg(\frac n2,\frac k2\bigg)$ \\
			& $=$ & $\ds \sum_{j=1}^{i-1}(-1)^{i+j+1} L_2(n,j)+ (-1)^{i+1}\!\!\sum_{\substack{k<i\\4|(n+k+2)}}\! L_2 \bigg(\frac n2,\frac k2\bigg).$\\ 
	\end{tabular}
	\vspace{6pt}
	
A similar argument holds when $n$ is even and  $n+i-1 \not\equiv 2 \!\pmod 4$. In this case, the rightmost sum does not acquire an extra term when rewriting it as a sum over $k<i$. 
\end{proof}

\section{Cycles in length 3 vector grid classes}\label{section 3-vector}

In this section, we enumerate cyclic permutations in each vector grid class of length 3. In Section \ref{sec:1and2}, we enumerate $\C_n^{+++}$ and $\C_n^{---}$, in Section \ref{SECTION 34}, we enumerate $\C_n^{+-+}$ and $\C_n^{-+-}$, and in Section \ref{SECTION 5678}, we enumerate $\C_n^{++-}$, $\C_n^{+--}$, $\C_n^{-++}$, and $\C_n^{--+}$. Finally, in Section~\ref{SECTION WILF}, we determine Wilf-equivalence classes for the length 3 vector grid classes.

\subsection{Enumerating $\mathcal C_n^{+++}$ and $\mathcal C_n^{---}$}\label{sec:1and2}

In Theorem \ref{thm sigma1 and sigma 2}, we enumerate both $\C_n^{+++}$ and $\C_n^{---}$. 
\begin{theorem}\label{thm sigma1 and sigma 2}
If $n\geq 1$, then $$c_{+++}(n) = L_3(n) - nL_2(n),$$ and if $n\geq 3$, then $$c_{---}(n) = L_3^*(n;---) - nL_2^*(n;--).$$
\end{theorem}
\begin{proof}
First, let us find $c_{+++}(n)$. The cases when $n$ is 1 or 2 are easily checked. Suppose $n\geq 3$. Notice that $c_{+++}(n) = |\C_n^{+++}\setminus \C_n^{++}|+|\C_n^{++}|$. We know also from Theorem \ref{thm: ++ and --} that $|\C_n^{++}|=L_2(n)$. We also know that permutations in $\C_n^{++}$ have exactly one descent, since the only permutation with fewer descents is the increasing permutation, which is not cyclic when $n\geq 2$. 

If $\pi \in \C_n^{+++}\setminus \C_n^{++}$, then $\pi$ must have exactly 2 descents. Suppose the descent set is $\{i,j\}$, with $i<j$. Then we must have that the $+++$-segmentation of $\pi$ is $(0,i,j,n)$. Thus, there is only one $+++$-segmentation. If $\pi\in\C_n^{++}$, then $\pi$ has exactly one descent. If the descent is at position $i$, then the $+++$-segmentation is $(0, e_1, e_2, n)$, where $\{e_1, e_2\} = \{i, j\}$ for any $j\in\{0,1,2,\ldots, n\}$. 
Thus, by Theorem \ref{BIJECTION BETWEEN WORDS AND SIGSEG}, we obtain that $$L_3^*(n;+++) = |\C_n^{+++}\setminus \C_n^{++}| + (n+1)|\C_n^{++}|.$$ Taken with the formula for $c_{+++}(n)$ in the second sentence of this proof and the fact that $L_3^*(n;\sigma) = L_3(n)$, we obtain the result $c_{+++}(n) = L_3(n) - nL_2(n)$.
The formula for $c_{---}(n)$ is found similarly.
\end{proof}

As stated in the introduction, the theorem above is a special case of theorems that appear in \cite{gessel} and \cite{ARCHERELIZALDE2014}. The general theorem is stated below. Let $n\geq1$ and $C_n(t) = \sum_{\pi\in\C_n} t^{\text{des}(\pi)},$ where $\text{des}(\pi)$ is the number of descents of $\pi$. 
\begin{theorem}\emph{\cite[Theorem 6.1]{gessel}}
If $n\geq 1$, then
$$
\dfrac{C_n(t)}{(1-t)^{n+1}} = \frac{1}{n} \sum_{d|n} \mu(d) \sum_{m=0}^{\infty} m^{n/d} t^m.
$$
\end{theorem}
%, proven using techniques similar to the ones used in this paper. 

Notice that Theorem~\ref{thm sigma1 and sigma 2} implies that the signatures $+++$ and $---$ are weakly cyc-Wilf-equivalent since when $n\not\equiv 2\!\pmod 4$, we must have $L_2^*(n;--) = L_2(n)$ and $L_3^*(n;---) = L_3(n)$. Clearly, the second statement of Theorem~\ref{thm sigma1 and sigma 2} is equivalent to $$
c_{---}(n) = \begin{cases}
 L_3(n) - nL_2(n) & \text{ when $n\not\equiv 2\!\!\!\!\pmod 4$} \\ 
 L_3(n) +L_3(\frac{n}{2})- n(L_2(n)+L_2(\frac{n}{2})) & \text{ when $n\equiv 2\!\!\!\!\pmod 4$}. 
\end{cases}
$$
It is straightforward to see that in the case when $n\equiv 2\!\pmod 4$, we have that $$d_3(n):=c_{---}(n)-c_{+++}(n) = L_3\bigg(\frac{n}{2}\bigg)-nL_2\bigg(\frac{n}{2}\bigg).$$ A table of these values can be seen in Table \ref{TABLE c1 c2}.

	\begin{figure}[h]
		\centering
	\begin{tabular}{|c|c|c|c|c|c|c|c|c|c|}
	\hline
	$n$ & 6 & 10 & 14 & 18 & 22 & 26 & 30& 34\\ \hline
	 $d_3(n)$ & $-4$ & $-12$ & 60 & 1176 & 12012 & 106260& 891116 & 7334340\\ \hline
	\end{tabular} 
	\captionof{table}{A table of the first eight values of $d_3:=c_{---}(n)-c_{+++}(n)$ for $n$ satisfying $n>2$ and $n\equiv 2 \!\pmod 4$.}
	\label{TABLE c1 c2}
	\end{figure}
	%	\todo{finish table of values (with OEIS entries)}
%%%%%%%%%%%%%%%%%%%%%%%%%%%%%%%%%%%%%%%%%
%%%%%%%%%%%%%%%%%%%%%%%%%%%%%%%%%%%%%%%%%

\subsection{Enumerating $\mathcal C_n^{+-+}$ and $\mathcal C_n^{-+-}$}\label{SECTION 34}

In this section, we enumerate the set of cyclic permutations in the $+-+$-class and the $-+-$-class. Recall that we define a \textit{peak} of the permutation $\pi\in\S_n$ to be $i\in[n]$ so that $\pi_{i-1}<\pi_i$ and $\pi_i>\pi_{i+1}$ (where $\pi_0=\pi_{n+1}:=0$). Similarly, we define a \textit{valley} of the permutation $\pi\in\S_n$ to be $i\in[n]$ so that $\pi_{i-1}>\pi_i$ and $\pi_i<\pi_{i+1}$ (where $\pi_0=\pi_{n+1}:=n+1$). For example, the permutation $4156732$ has peaks at positions 1 and 5 and has valleys at positions 2 and 7. 

\begin{lemma}
 Let $n\geq 3$. For any permutation $\pi\in\C_n^{+-+}$, there is a unique choice of $i,j \in [n]$ with $i<j$ so that $i$ is a peak and $j$ is a valley.
\end{lemma}

\begin{proof}
Suppose that $\pi\in \C_n^{+-+}$.  If $\e=(0, e_1, e_2, n)$ is a $+-+$-segmentation of $\pi$, then $\pi_1\pi_2\ldots\pi_{e_1}$ is increasing, $\pi_{e_1+1}\ldots \pi_{e_2}$ is decreasing, and $\pi_{e_2+1}\ldots \pi_n$ is increasing. If we have strict inequalities $0<e_1<e_2<n$, then clearly, $i=e_1$ if $\pi_{e_1}>\pi_{e_1+1}$ and $i=e_1+1$ otherwise, and $j=e_2$ if $\pi_{e_2}>\pi_{e_2+1}$ and $j=e_2+1$ otherwise. In this case, we do not have $i=j$ since that would imply that $\pi$ is the increasing permutation, which is not cyclic. Similarly, at least two of the inequalities $0\leq e_1\leq e_2\leq n$ are strict inequalities, otherwise $\pi$ would be the increasing or decreasing permutation.
 
If we have that $e_1=0$, then $\pi \in \C_n^{-+}$. In this case, we can take $i=1$ and $j$ as above. Again, we do not have $i=j$ as that would imply $\pi$ is strictly increasing. If we have that $e_2=n$, we can take $j=n$ and $i$ as above. We similarly do not have $i=j$. Finally, if $e_1=e_2$, then $i=e_1$ and $j=e_1+1$. 
\end{proof}

For example, consider the permutation $\pi = 356894127 \in \C_9^{+-+}$. There are two peaks (namely 5 and 9) and two valleys (namely 1 and 7), but there is a unique pair $i=5$ and $j=7$ satisfying the condition that $i<j$, where $i$ is a peak and $j$ is a valley.

\begin{theorem}\label{COUNTING SIGMA3}
For $n \geq 3$, $c_{+-+}(n)=\dfrac{L_3^*(n;+-+)}{4}$.
\end{theorem}

\begin{proof} Let $\sigma:=+-+$ for the duration of this proof.
By Theorem~\ref{BIJECTION BETWEEN WORDS AND SIGSEG}, we have
	\[L_3^*(n;\sigma)= \sum_{\pi \in \mathcal C_n^{+-+}} |\{\e: \e \text{ is a }\sigma\text{-segmentation of } \pi\}|.\]
If $\pi \in \mathcal C_n^{\sigma}$, then there exists some unique $i,j \in [n]$ with $i<j$ such that $i$ is a peak and $j$ is a valley.  Suppose $\e=(0, e_1, e_2, n)$ is a $\sigma$-segmentation of $\pi$. Certainly, $e_1\leq i$.  If $e_1<i-1$, then $e_2 \in \{e_1,e_1+1\}$ and thus $e_2<i$.  Since $i$ is a peak, it is followed by a descent and thus the $\sigma$-segmentation constructed is not valid. Therefore, we must have that $e_1\in\{i-1,i\}$. For similar reasons, $e_2 \in \{j-1,j\}$ and so the following list of four $\sigma$-segmentations of $\pi$ is complete:
	\begin{enumerate}[(a)]
		\item $(0, i-1, j-1, n)$,
		\item $(0, i-1, j, n)$,
		\item $(0, i, j-1, n)$,
		\item $(0, i, j, n)$.
	\end{enumerate} 
  Therefore, $L_3^*(n;\sigma)=4c_{\sigma}(n)$ and the result follows.
\end{proof}

Notice that for a permutation $\pi\in\C_n^{-+-}$, there is a unique $i,j \in [n]$ with $i<j$ so that $i$ is a valley and $j$ is a peak. 
Thus, we omit the proof of the following theorem due to its similarities to the proof of Theorem~\ref{COUNTING SIGMA3}.

\begin{theorem}\label{COUNTING SIGMA4}
For $n \geq 3$, $c_{-+-}(n)=\dfrac{L_3^*(n;-+-)}{4}.$
\end{theorem}

%%%%%%%%%%%%%%%%%%%%%%%%%%%%%%%%%%%%%%%%%
%%%%%%%%%%%%%%%%%%%%%%%%%%%%%%%%%%%%%%%%%

\subsection{Enumerating $\C_n^{++-}$, $\C_n^{+--}$, $\C_n^{-++}$, and $\C_n^{--+}$.}\label{SECTION 5678}

In this section, we enumerate the set $\mathcal C_n^{\sigma}$ for the four remaining signatures $\sigma\in\{+,-\}^3$. 
\begin{theorem}\label{COUNT SIGMA5}
Let $\Lambda(n,i)$ be the value given in Lemma~\ref{UP DOWN PEAK}. If $n\geq 2$, then
	\[c_{++-}(n)=\frac{L^*_3(n;++-)}{2}-\sum_{i=2}^{n-1}i\cdot\Lambda(n,i).\]
	%	\[c_{++-}(n)=\Lambda(n)+\frac{L^*_3(n;++-)}{2}-\sum_{i=2}^{n-1}(i+1)\Lambda(n,i).\]
\end{theorem}

\begin{proof}
Let $\sigma:=++-$ for the duration of this proof. If a permutation is an element of $\mathcal C_n^{++-}\setminus \C_n^{+-}$, then it has two $\sigma$-segmentations.  As a result, the equality
	\[c_{\sigma}(n) = |\mathcal C_n^{++-}\setminus \C_n^{+-}|+\Lambda(n)\]
and Theorem~\ref{BIJECTION BETWEEN WORDS AND SIGSEG} imply
	\begin{eqnarray*}
		L_3^*(n;\sigma) & = & \sum_{\pi \in\mathcal C_n^{\sigma}} |\{\e: \e\text{ is a }\sigma\text{-segmentation of } \pi\}|  \\
			& = & 2\big(c_{\sigma}(n)-\Lambda(n)\big) + \sum_{\pi \in C_n^{+-}} |\{\e: \e\text{ is a }\sigma\text{-segmentation of } \pi\}|.
	\end{eqnarray*}
Notice there are $2(i+1)$ $\sigma$-segmentations of each $\pi \in \C_n^{+-}$ with a peak at $i$.  Specifically, each of these $\sigma$-segmentations will have one of the following forms:
	\begin{enumerate}[(a)]
		\item $(0, j, i-1, n)$ for each $j \in \{0,1,\ldots,i-1\}$,
		\item $(0, j, i, n)$ for each $j \in \{0,1,\ldots,i\}$,
		\item $(0, i, i+1, n)$.
	\end{enumerate} 
A similar argument to the one in the proof of Theorem~\ref{COUNTING SIGMA3} shows that these are the only $\sigma$-segmentations of $\pi$.

Let $\C_{n,i}^{+-}$ denote the set of permutations $\pi \in \C_n^{+-}$ with a peak in position $i$. Then
	\begin{eqnarray*}\sum_{\pi \in \C_n^{+-}} |\{\e: \e\text{ is a }\sigma\text{-segmentation of }\pi\}| & =&   \sum_{i=2}^{n-1}  \sum_{\pi \in \C_{n,i}^{+-}} 2(i+1)\\ &=&   \sum_{i=2}^{n-1}  2(i+1)\Lambda(n,i)\end{eqnarray*}
and
	\[L_3^*(n;++-) = 2c_{++-}(n)-2\Lambda(n)+\sum_{i=2}^{n-1}  2(i+1)\Lambda(n,i).\]
Since $\Lambda(n) = \sum \Lambda(n,i)$,  the theorem follows from the equality above.
\end{proof}

\begin{theorem} Let $\Lambda(n,i)$ be the value given in Lemma~\ref{UP DOWN PEAK}. If $n\geq 2$, then
	\[c_{+--}(n)=\frac{L^*_3(n;+--)}{2}-\sum_{i=2}^{n-1}(n-i+1)\Lambda(n,i).\]
\end{theorem}
\begin{proof} Let $\sigma:=+--$ for the duration of this proof. 
As in the proof of Theorem~\ref{COUNT SIGMA5}, there are exactly two $\sigma$-segmentations of $\pi$ whenever $\pi \in \C_n^{+--}\setminus\C_n^{+-}$. Furthermore, for $\pi \in \C_n^{+-}$ with a peak at $i$, a similar argument to the one in the proof above shows that the $\sigma$-segmentations of each of these permutations will have one of the following forms:
	\begin{enumerate}[(a)]
		\item $(0, i-1, j, n)$ for each $j \in \{i-1, i, \ldots, n\}$,
		\item $(0, i, j, n)$ for each $j \in \{i, i+1, \ldots, n\}$,
		\item $(0, i-2, i-1, n)$.
	\end{enumerate} 
Thus
	\begin{eqnarray*}
		\sum_{\pi \in \C_n^{+-}} |\{\e: \e\text{ is a }\sigma\text{-segmentation of } \pi\}|  & = &  \sum_{i=2}^{n-1}  \sum_{\pi \in \C_{n,i}^{+-}} 2(n+i+2) \\
		& = &  \sum_{i=2}^{n-1}  2(n+i+2)\Lambda(n,i)
	\end{eqnarray*}
and the result now follows.
\end{proof}

The following corollaries are immediate results of Proposition~\ref{REVERSE SAME} and Lemma~\ref{REVERSE WORD}.
\begin{cor} The signatures $++-$ and $-++$ are cyc-Wilf-equivalent, and thus if $n\geq 2$, 
$$\displaystyle c_{-++}(n) = \frac{L^*_3(n;-++)}{2}-\sum_{i=2}^{n-1}i\cdot\Lambda(n,i).$$
\end{cor}

\begin{cor}
The signatures $+--$ and $--+$ are cyc-Wilf-equivalent, and thus if $n\geq 2$, 
$$\displaystyle c_{--+}(n) = \frac{L^*_3(n;--+)}{2}-\sum_{i=2}^{n-1}(n-i+1)\Lambda(n,i).$$
\end{cor}

\subsection{cyc-Wilf-equivalence}\label{SECTION WILF}

In this section, we discuss the equivalence classes of the cyclic permutations in the length 3 vector grid classes. Recall two classes are cyc-Wilf-equivalent if the sets of cyclic permutations in each class are equinumerous, and we say two classes are weakly cyc-Wilf-equivalent if they are equinumerous when $n\not\equiv 2 \!\pmod 4$.

\begin{theorem}\label{thm sigma3 equiv sigma4}
For $n \geq 1$, $c_{+-+}(n)= c_{-+-}(n)$ and thus the signature $+-+$ is cyc-Wilf-equivalent to the signature $-+-$. 
\end{theorem}

\begin{proof}
For $n=1$ and $n=2$, this is easily checked and thus we assume $n\geq 3$. By Theorems~\ref{COUNTING SIGMA3} and~\ref{COUNTING SIGMA4}, it is enough to show $L^o_3(\frac{n}{2};+-+)=L^o_3(\frac{n}{2};-+-)$. Since these are both 0 when $n$ is odd, we need only consider the cases when $n$ is even. It suffices to show that the number of necklaces $[s]$ of length $\frac{n}{2}$ when $o_{+-+}(s)$ is odd is equal to the number of necklaces $[s']$ of length $\frac{n}{2}$ when $o_{-+-}(s)$ is odd.
Note that for a necklace $[s]$ with $s_i \in \{0,1,2\}$ of length $\frac{n}{2}$, we have $o_{+-+}(s)=|\{i \in [n] : s_i =1\}|$ and $o_{-+-}(s)=|\{i \in [n] : s_i \neq1\}|$ by definition.   We proceed by cases.

\textit{Case 1.} When $n \equiv 0 \!\pmod 4$, $\frac n2$ is even, and thus for any necklace $[s]$ of length $\frac n2$, $o_{+-+}(s)$ is odd if and only if $o_{-+-}(s)$ is odd. Therefore, $L^o_3(\frac{n}{2};+-+)=L^o_3(\frac{n}{2};-+-)$.

\textit{Case 2.} When $n\equiv 2 \!\pmod 4$, $\frac n2$ is odd and we will show using M\"{o}bius inversion that exactly half of necklaces of length $\frac n2$ have $o_{+-+}(s)$ odd. This in turn implies that there are equally many necklaces with $o_{+-+}(s)$ odd as there are with $o_{+-+}(s)$ even (or equivalently, $o_{-+-}(s)$ odd). 

Take $m := \frac n2$ for simplicity of notation. Let $a(m)$ be the number of ternary primitive words with $o_{+-+}(s)$ odd and let $b(m)$ be the number of ternary words (which are not necessarily primitive) with $o_{+-+}(s)$ odd. It is easily checked that $b(m) = \frac{3^n-1}{2}$. Also, for any $m$, 
	$$b(m) = \sum_{\substack{d|m\\ d\ \mathrm{odd}}} a(m/d).$$
Since $m$ is odd, all its divisors are also odd, so we can write $$b(m) = \sum_{d\mid m} a(m/d)= \sum_{d\mid m} a(d).$$ Using M\"{o}bius inversion, we rewrite this sum as $$ a(m) = \sum_{d\mid m} \mu(d)b(m/d).$$ 
This sum becomes $$a(m) = \frac{1}{2} \sum_{d\mid m} \mu(d) (3^{n/d}-1) = \frac{1}{2} \sum_{d\mid m} \mu(d) 3^{n/d},$$ where the equality follows from the well-known fact that $\sum \mu(d) = 0$ when the sum is taken over all divisors $d\mid m$ for any integer $m\geq 2$. It follows that $$L^o_3(m;+-+) = \frac{1}{m}a(m) = \frac{1}{2m}\sum_{d\mid m} \mu(d) 3^{n/d}$$ which is exactly half of $L_3(m)$. The result follows.
\end{proof}

\begin{theorem}
If $n\geq 1$ and $n\not\equiv 2 \!\pmod 4$, then $c_{++-}(n) = c_{+--}(n)$ and thus the signatures $++-$ and $+--$ are weakly cyc-Wilf-equivalent.
\end{theorem}

\begin{proof}
It is enough to show that when $n\not\equiv 2 \!\pmod 4$, $$\sum_{i=2}^{n-1}i\cdot\Lambda(n,i) = \sum_{i=2}^{n-1}(n-i+1)\Lambda(n,i) \quad \text{ and } \quad L^*_3(n;++-) = L^*_3(n;+--).$$
Setting $i:=n-j+1$ and using Lemma \ref{SYMMETRY LEMMA}, we obtain the equalities $$\sum_{i=2}^{n-1}i\cdot \Lambda(n,i) = \sum_{j=2}^{n-1}(n-j+1)\Lambda(n,n-j-1) = \sum_{j=2}^{n-1}(n-j+1)\Lambda(n,j).$$
If $n$ is odd, it is clear that $L^*_3(n;++-) = L^*_3(n;+--)$. When $n\equiv 0 \!\pmod 4$, it is enough to show that $L^o_3(\frac{n}{2}; ++-) = L^o_3(\frac{n}{2}; +--).$ Observe $L^o_3(\frac{n}{2}; ++-)$ is the size of the set of primitive necklaces of length $\frac{n}{2}$ which have an odd number of 2's. Since $\frac{n}{2}$ is even,  the number of 0's and 1's must be odd as well. Therefore, $L^o_3(\frac{n}{2}; ++-) = L^o_3(\frac{n}{2}; --+)$  and thus the equality $L^o_3(\frac{n}{2}; ++-) = L^o_3(\frac{n}{2}; +--)$ follows from Lemma~\ref{REVERSE WORD}.
\end{proof}

%\todo{ref wants us to compare $c_{++-}(n)$ and $c_{+--}(n)$ here too...}

%It is straightforward to see that in the case when $n\equiv 2\!\pmod 4$, we have that $$c_{+--}(n)-c_{++-}(n) = L_3\big(\frac{n}{2}\big)-nL_2\big(\frac{n}{2}\big).$$ A table of these values can be seen in Table \ref{TABLE c5 c6}.
%
%	\begin{figure}[h]
%		\centering
%	\begin{tabular}{|c|c|c|c|c|c|c|c|c|c|}
%	\hline
%	$n$ & 6 & 10 & 14 & 18 & 22 & 26 & 30& 34\\ \hline
%	 $c_{+--}(n)-c_{++-}(n)$ & $-4$ & $-12$ & 60 & 1176 & 12012 & 106260& 891116 & 7334340\\ \hline
%	\end{tabular} 
%	\captionof{table}{}
%	\label{TABLE c5 c6}
%	\end{figure}

\section{Alternating grid classes}\label{SECTION ALTERNATING}

In this section, we prove cyc-Wilf-equivalence for a more general set of grid classes.
Define the \textit{$k$-th up-down alternating grid class} to have signature $\sigma_{alt}^{+k}=+-+-\cdots \in \{+,-\}^k$ and the \textit{$k$-th down-up alternating grid class} to have signature $\sigma_{alt}^{-k}=-+-+\cdots\in\{+,-\}^k$. % If an alternating $\sigma$ is of length $k$, this means that a permutation in the $\sigma$-class has at most $k+1$ peaks and valleys (including those that occur at the first and last position in the permutation). 
In this section, we show that the number of cyclic permutations in the up-down alternating grid class is equal to the number of cyclic permutations in the down-up alternating grid class (i.e., they are cyc-Wilf-equivalent). When $k$ is even, this is automatically true by Proposition~\ref{REVERSE SAME}, so it remains to show this is true when $k$ is odd.
%Let $L^+_k(n)$ denote short hand for $L^*_k(n;\sigma_a^+)$ and let $L^-_k(n)$ denote short hand for $L_k^*(n;\sigma_a^-)$. 
%For a necklace $s$ on $k$ letters $\{0,1,2,\ldots, k-1\}$ of length $n$, r$o(s)$ denote the size of the set $\{i \in [n] : s_i \text{ is odd}\}$ and let $e(s)$ denote the size of the set $\{i \in [n] : s_i \text{ is even}\}$.

\begin{lemma}\label{lem:alt}
For all odd $k\geq 3$ and $n\geq 3$, $L^*_k(n;\sigma_{alt}^{+k}) = L_k^*(n;\sigma_{alt}^{-k}).$
\end{lemma}
The proof of this lemma is exactly the proof of Theorem \ref{thm sigma3 equiv sigma4}, modified to enumerate $k$-ary necklaces of length $\frac{n}{2}$ with $o_{\sigma_{alt}^{+k}}(s)$ odd. Again, one may use  M\"{o}bius inversion to show that when $\frac{n}{2}$ is odd, exactly half of the primitive $k$-ary necklaces of length $\frac{n}{2}$ have $o_{\sigma_{alt}^{+k}}(s)$ odd.

\begin{theorem}
For all $k\geq 2$, the signatures $\sigma_{alt}^{+k}$ and $\sigma_{alt}^{-k}$ are cyc-Wilf-equivalent.
\end{theorem}
\begin{proof}
As stated above, this is automatically true if the length of the signatures is even. Assume the length $k$ is odd. The theorem follows inductively from Lemma~\ref{lem:alt} and Theorem~\ref{BIJECTION BETWEEN WORDS AND SIGSEG}.
\end{proof}

%%%%%%%%%%%%%%%%%%%%%%%%%%%%%%%%%%%%%%%%
%%%%%%%%%%%%%%%%%%%%%%%%%%%%%%%%%%%%%%%%
\section{Discussion and conjectures} \label{SECTION CONJECTURES}

Recall that for $\sigma \in \{+,-\}^k$, we define the complementary signature as $\sigma_c = \sigma^0_c\sigma^1_c\ldots \sigma^{k-1}_c$ with $\sigma^i_c = +$ if $\sigma^i=-$ and $\sigma^i_c = -$ if $\sigma^i=+$.
\begin{conjecture} For all $k\geq 2$ and $\sigma \in\{+,-\}^k$, 
$\sigma$ is weakly cyc-Wilf-equivalent to $\sigma_c$.
\end{conjecture}

A signature $\sigma$ has $k$ \textit{corners} if there are $k$ changes in sign. For example, $\sigma =+++$ has 0 corners, $\sigma = +-+-+-$ has 5 corners, and $\sigma = ++---+$ has 2 corners. The following conjecture subsumes the one above since a signature and its complement have the same number of corners.
\begin{conjecture} 
For all $k\geq 2$ and $\sigma \in\{+,-\}^k$, signatures with $k$ corners are all weakly cyc-Wilf-equivalent. 
\end{conjecture}

%Define an alternating signature of size $k$ to be $\sigma$ so that either $\sigma^i=+$ for all even $i$ and $\sigma^i=-$ for all odd $i$ or $\sigma^i=+$ for all odd $i$ and $\sigma^i=-$ for all even $i$. If $k$ is even, then the below statement is automatically true by Proposition~\ref{REVERSE SAME}, so it remains to show this is true when $k$ is odd.
%
%\begin{conjecture} For all $k\geq 2$, the alternating signatures of size $k$ are all cyc-Wilf-equivalent. 
%\end{conjecture}
%
%

More generally, one could use Theorem~\ref{BIJECTION BETWEEN WORDS AND SIGSEG} to enumerate the set of cyclic permutations in other families of grid classes. Furthermore, one could generalize this theorem to a bijection between $\sigma$-segmentations of permutations with other cycle types and multisets of necklaces in order to enumerate permutations in grid classes by their cycle type. Some progress towards this for $\sigma = +-$ and $\sigma = +^k$ can be found in \cite{Thibon} and \cite{gessel}.  

%%%%%%%%%%%%%%%%%%%%%%%%%%%%%%%%%%%%%%%%%
%%%%%%%%%%%%%%%%%%%%%%%%%%%%%%%%%%%%%%%%%

%\subsubsection*{Acknowledgements.}
%The authors would like to thank the anonymous referee for their suggestions in modifying this paper.

\bibliographystyle{plain}
\bibliography{References_LK}

\end{document}